\documentclass[11pt, one side,emlines]{amsart}
\usepackage{amssymb,latexsym,xy,eucal,mathrsfs, graphicx, tikz, amsmath, caption}
\textwidth=15cm \textheight=24cm \theoremstyle{plain}
\newtheorem{lemma}{Lemma}[section]
\newtheorem{thm}[lemma]{Theorem}

\newtheorem{proposition}[lemma]{Proposition}

\theoremstyle{definition}

\newtheorem{definition}[lemma]{Definition}
\def\prfname{{\it Proof of Theorem }}

\numberwithin{equation}{section} 
\begin{document}
	
	\title[Pendant 3-tree Connectivity of Augmented Cubes]{Pendant 3-tree Connectivity of Augmented Cubes }
	\author{S. A. Mane, S. A. Kandekar\\ }
	\address{\rm Center for Advanced Studies in Mathematics, Department of Mathematics, Savitribai Phule Pune University, Pune 411007, M.S., INDIA.}
	\email{\emph{manesmruti@yahoo.com; smitakandekar54@gmail.com}}
	\maketitle
	
	\baselineskip 20 truept
	\begin{abstract}
	The Steiner tree problem in graphs has applications in network design or circuit layout. Given a set $S$ of vertices, $|S| \geq 2,$ a tree connecting all vertices of $S$ is called an $S$-Steiner tree (tree connecting $S$). The reliability of a network $G$ to connect any $S$ vertices ($|S|$ number of vertices) in $G$ can be measure by this parameter.   
	For an $S$-Steiner tree, if the degree of each vertex in $S$ is equal to one, then that tree is called a pendant S-Steiner tree.
	Two pendant $S$-Steiner trees $T$ and $T'$ are said to be internally disjoint if $E(T) \cap E(T') = \emptyset$ and $V(T) \cap V(T') = S.$ The local pendant tree-connectivity $\tau_{G}(S)$ is the maximum number of internally disjoint pendant $S$-Steiner trees in $G.$
	For an integer $k$ with $2 \leq k \leq n,$ the pendant k-tree-connectivity is defined as $\tau_{k}(G) = min\{ \tau_{G}(S) : S \subseteq V(G), |S| = k\}.$
	In this paper, we study the pendant $3$-tree connectivity of Augmented cubes which are modification of hypercubes invented to increase the connectivity and decrease the diameter hence superior to hypercubes. We
	show that $\tau_3(AQ_n) = 2n-3.$ , which attains the upper bound of $\tau_3(G)$ given by Hager, for $G = AQ_n$.
	\end{abstract}
	augmented cube    
	\vskip.2cm
	\noindent
	
	\section{\textbf{Introduction}}
	
In networks which are usually denoted as a graph $G$. A graph $G = (V, E)$ be a finite, undirected graph. In network design applications the edges from the edge set $E$ of the input graph $G$ typically correspond to the communication link connecting two processors where vertices from the vertex set $V$ are representing processors. 

Many useful topologies have been proposed to balance cost parameters and performance. Additionally, if the topology is the Cayley graph then it is more preferable as the Cayley graph has some attractive properties to design interconnection networks. The hypercube \cite{we} denoted by $Q_n$ is one of the most popular topologies and has been studied for parallel networks. 

Augmented cubes are derivatives of hypercubes (proposed by Choudam and Sunitha \cite{cs}) with good geometric features that retain some favorable properties of the hypercubes (since $Q_n \subset AQ_n$), such as vertex symmetry, routing and broadcasting procedures with linear time complexity. An $n-$dimensional augmented cube $AQ_n$ can be formed as an extension of $Q_n$ by adding some links. For any positive integer $n$, $AQ_n$ is $(2n-1)$-regular and  $(2n-1)$-connected (except $n = 3$) graph with $2^n$ vertices. The diameter of $AQ_n$ is about half of that of $Q_n$. Several of the graphs, which model the algorithms falling under the divide-and-conquer paradigm, are binary or binomial trees. Choudum and Sunita \cite{cs} showed that the augmented cube of dimension $n$ contains two edge-disjoint complete binary trees on $2^n -1$ vertices both rooted at the same vertex, two edge-disjoint spanning binomial trees, and all the $k$-cycles, $3 \leq k \leq 2^n.$ These properties are not shared by hypercubes and their other variations. As well not all the variations of hypercubes are Cayley graphs but augmented cubes are Cayley graphs.  Considering all these properties, $AQ_n$ can potentially be a good alternative to a hypercube. 

Steiner tree is an attractive topic for researchers in combinatorial optimization and computer science. Given a set $S$ of vertices, $|S| \geq 2,$ a tree connecting all vertices of $S$ is called an $S$-Steiner tree (tree connecting $S$). The reliability of a network $G$ to connect any $S$ vertices ($|S|$ number of vertices) in $G$ can be measure by this parameter.   
For an $S$-Steiner tree, if the degree of each vertex in $S$ is equal to one, then that tree is called a pendant $S-$Steiner tree.
Two pendant $S$-Steiner trees $T$ and $T'$ are said to be internally disjoint if $E(T) \cap E(T') = \emptyset$ and $V(T) \cap V(T') = S.$ The local pendant tree-connectivity $\tau_{G}(S)$ is the maximum number of internally disjoint pendant $S$-Steiner trees in $G.$
For an integer $k$ with $2 \leq k \leq n,$ the pendant k-tree-connectivity is defined as $\tau_{k}(G) = min\{ \tau_{G}(S) : S \subseteq V(G), |S| = k\}.$ The concept of pendant-tree connectivity was introduced by Hager \cite{hg1} in 1985, which
is specialization of generalized connectivity (or k-tree-connectivity defined by Chartrand \cite{chart1}) but a generalization of classical connectivity. In the definition of pendant tree connectivity if we remove restriction on degree of each vertex in $S$ to be one then it become generalize connectivity. Thus the only difference between generalize connectivity and pendant tree connectivity is on the degree of each vertex in $S$. The generalized $2$-connectivity $k_2(G)$ of $G$ is exactly the connectivity $k(G)$. Moreover, $k_n(G)$ is exactly the spanning tree packing number of $G.$ Therefore, the generalized connectivity is a common generalization
of the classical connectivity and spanning tree packing number. The research about $S-$Steiner trees, spanning tree packing number, generalized connectivity and pendant tree connectivity
 of graphs plays a key role in effective information transportation
in terms of parallel routing design for large-scale networks, see \cite{chart2, Chen, clm, gu, li6, lima, li5, li4, li3, lishi, lin, li2, li1, li7, Yi, ss, mao1, mao2, oe, ssp, wei, Zha, zha1, zha2, Shu}.

As a link between discrete mathematics and theoretical computer science, the algorithmic graph theory has become increasingly important in recent years. Proving by induction method is part of algorithmic graph theory. After proving generalized connectivity of augmented cubes \cite{sak}
by induction in quite easy manner here in this paper, we study the pendant $3$-tree connectivity of Augmented cubes which are modification of hypercubes invented to increase the connectivity and decrease the diameter hence superior to hypercubes. We
show that $\tau_3(AQ_n) = 2n-3.$ , which attains the upper bound of $\tau_3(G)$ given by Hager \cite{hg1}, for $G = AQ_n$.
	
\section{ \textbf{Preliminaries}}
The $n$-dimensional augmented cube is denoted by $AQ_n, n\geq 1.$ It is a graph with vertex set $\{0, 1\}^n$, the set of all binary $n$-tuples. It is defined recursively as follows.\\
$AQ_1$ is the complete graph $K_2$ with vertex set $\{0, 1\}.$ For $n \geq 2,~ AQ_n$ is obtained from two copies of $AQ_{n-1},$ denoted by $AQ^0_{n-1}$ and $AQ^1_{n-1},$ and adding $2^n$ edges between them as follows.\\
Let $V(AQ^0_{n-1}) = \{ 0x_1x_2...x_{n-1} \colon x_i = 0~~ \rm{or}~~ 1\}$ and $V(AQ^1_{n-1}) = \{1y_1y_2...y_{n-1} \colon y_i = 0~~ \rm{or}~~ 1\}.$
A vertex $x = 0x_1x_2...x_{n-1}$ of $AQ^0_{n-1}$ is joined to a vertex $y = 1y_1y_2...y_{n-1}$ of $AQ^1_{n-1}$ if and only if either 
\begin{enumerate}
	\item $x_i = y_i$ for $1 \leq i \leq n-1,$ in this case the edge $xy$ is called a hypercube edge and we set $y = x^h$ or
	
	\item $x_i = \overline{y_i}$ for $1 \leq i \leq n-1,$ in this case the edge $xy$ is called a complementary edge and we set $y = x^c.$
\end{enumerate}

Let $E_n^h$ and $E_n^c$ be the set of hypercube edges and complementary edges, respectively used to construct $AQ_n$ from two copies of $AQ_{n-1}.$ Then $E_n^h$ and $E_n^c$ are perfect matchings of $AQ_n$ and further,  $AQ_n = AQ^0_{n-1} \cup AQ^1_{n-1} \cup E_n^h \cup E_n^c.$ The augmented cubes of dimensions 1, 2 and 3 are shown in Figure $1.$ 

\hspace*{-0.5cm}
\begin{tikzpicture}[scale=1]
\draw [style= thick][blue] (0,0)--(0,2);
\draw [fill=black] (0,0) circle  (.08)  node [left]  at (0,0) {$0$};
\draw [fill=black] (0,2) circle  (.08)  node [left]  at (0,2) {$1$};


\draw [style= thick][blue] (3,0)--(3,2);
\draw [style= thick][green] (3,0)--(5,0);
\draw [style= thick][blue] (5,0)--(5,2);
\draw [style= thick][red] (5,0)--(3,2);
\draw [style= thick][red] (3,0)--(5,2);
\draw [style= thick][green] (5,2)--(3,2);
\draw [fill=black] (3,0) circle  (.08)  node [left]  at (3,0) {$00$};
\draw [fill=black] (3,2) circle  (.08)  node [left]  at (3,2) {$01$};
\draw [fill=black] (5,0) circle  (.08)  node [right]  at (5,0) {$10$};
\draw [fill=black] (5,2) circle  (.08)  node [right]  at (5,2) {$11$};


\draw [style= thick][blue] (8,0)--(8,2);
\draw [style= thick][blue] (8,0)--(10,0);
\draw [style= thick][blue] (10,0)--(10,2);
\draw [style= thick][blue] (10,0)--(8,2);
\draw [style= thick][blue] (8,0)--(10,2);
\draw [style= thick][blue] (10,2)--(8,2);

\draw [style= thick][blue] (12,0)--(12,2);
\draw [style= thick][blue] (12,0)--(14,0);
\draw [style= thick][blue] (14,0)--(14,2);
\draw [style= thick][blue] (14,0)--(12,2);
\draw [style= thick][blue] (12,0)--(14,2);
\draw [style= thick][blue] (14,2)--(12,2);

\draw [style= thick][red] (8,0)--(14,2);
\draw [style= thick][red] (14,0)--(8,2);
\draw [style= thick][red] (10,2)--(12,0);
\draw [style= thick][red] (10,0)--(12,2);

\draw [style= thick][green] (8,0)..controls(10, -0.5)..(12,0);
\draw [style= thick][green] (10,0)..controls(12,-0.5 )..(14,0);
\draw [style= thick][green] (8,2)..controls(10,2.5 )..(12,2);
\draw [style= thick][green] (10,2)..controls(12,2.5 )..(14,2);

\draw [fill=black] (8,0) circle  (.08)  node [left]  at (8,0) {$000$};
\draw [fill=black] (8,2) circle  (.08)  node [left]  at (8,2) {$001$};
\draw [fill=black] (10,0) circle  (.08)  node [right]  at (10.1,0.05) {$010$};
\draw [fill=black] (10,2) circle  (.08)  node [right]  at (10.1,1.95) {$011$};

\draw [fill=black] (12,0) circle  (.08)  node [left]  at (11.9,0.05) {$100$};
\draw [fill=black] (12,2) circle  (.08)  node [left]  at (11.9,1.95) {$101$};
\draw [fill=black] (14,0) circle  (.08)  node [right]  at (14,0) {$110$};
\draw [fill=black] (14,2) circle  (.08)  node [right]  at (14,2) {$111$};

\end{tikzpicture}

\hspace{0.05in} {$AQ_1$}  \hspace{1.1in } {$AQ_2$} \hspace{2.2in } {$AQ_3$}

\hspace{1.0in }  {Figure 1: Augmented cubes of dimensions 1, 2 and 3.}\\


In $AQ_3,$ if $u_1 = 000, u_2 = 001, u_3 = 011, u_4 = 010$ and $v_1 =100, v_2 = 101, v_3 = 111, v_4 =110,$ then $E_3^h = \{u_iv_i : i = 1, 2, 3, 4\}$ and $E_3^c = \{ u_1v_3, u_2v_4, u_3v_1, u_4v_2\}.$

From the definition, it is clear that $AQ_n$ is a $(2n-1)$-regular graph on $2^n$ vertices. It is also known that $AQ_n$ is $(2n-1)$-connected and vertex-symmetric \cite{cs}. 	

\section[{Pendant 3-tree-connectivity of Augmented Cubes}]{Pendant 3-tree-connectivity of Augmented Cubes, $\mathbf{\tau_3(AQ_n)}$ }

With the concept of tree-connectivity, Hager \cite{hg1} also introduced another tree-connectivity parameter, called the {\it pendant tree-connectivity} of a graph in \cite{hg1}. In this section, we obtain the pendant 3-tree-connectivity of $AQ_n.$ Let us recall some definitions which are required for further discussion. For this, see \cite{li7}. 

\begin{definition}[\cite{hg1}]
	For an $S$-Steiner tree, if the degree of each vertex in $S$ is equal to one, then that tree is called a pendant S-Steiner tree.

\end{definition}
	
	Two pendant $S$-Steiner trees $T$ and $T'$ are said to be internally disjoint if $E(T) \cap E(T') = \emptyset$ and $V(T) \cap V(T') = S.$ For $S \subseteq V(G)$ and $|S| \geq 2,$ the local pendant tree-connectivity $\tau_{G}(S)$ is the maximum number of internally disjoint pendant $S$-Steiner trees in $G.$

\begin{definition}[\cite{hg1}]	
	For an integer $k$ with $2 \leq k \leq n,$ the pendant k-tree-connectivity is defined as $$\tau_{k}(G) = min\{ \tau_{G}(S) : S \subseteq V(G), |S| = k\}.$$
\end{definition}
\noindent	
	By convention, 
	$\tau_{k}(G) = 0,$ when $G$ is disconnected.	
	\begin{center}
		\hspace*{-1cm}
		\begin{tikzpicture}[scale=0.9]

		\draw [style = thick][green](0,0)--(2,0);
		\draw [style = thick][green](2,0)--(2,2);
		\draw [style = thick][green](2,0)--(0,2);
		\draw [style = thick][red](4,0)--(6,0);
		\draw [style = thick] [blue] (4,2)--(6,2);
				
		\draw [style = thick] [blue](0,0)--(6,2);
		\draw [style = thick][red](6,0)--(0,2);
		\draw [style = thick][red](4,0)--(2,2);
		\draw [style = thick] [red](0,0)..controls(2,-0.5)..(4,0);
		\draw [style = thick] [blue] (0,2)..controls(2,2.5)..(4,2);
		\draw [style = thick] [blue] (2,2)..controls(4, 2.5)..(6,2);

		\draw[fill= white] (0,0) circle (0.2);
		\draw[fill= black] (0,0) circle (0.1) node [below] at (0,-0.2) {\small$000$};
		\draw[fill= white] (0,2) circle (0.2);
		\draw[fill= black] (0,2) circle (0.1) node [above] at (0,2.2) {\small$001$};
		\draw[fill= black] (2,0) circle (.08);
		\draw[fill= white] (2,2) circle (0.2);
		\draw[fill= black] (2,2) circle (0.1) node [right] at (2.2,1.9) {\small$011$};
		\draw[fill= black] (4,0) circle (.08);
		\draw[fill= black] (4,2) circle (.08);
		\draw[fill= black] (6,0) circle (.08);
		\draw[fill= black] (6,2) circle (.08);
	
	\draw [style = thick][blue](8,0)--(8,2);
	\draw [style = thick][blue](8,0)--(10,0);
	\draw [style = thick][red](10,0)--(10,2);
	\draw [style = thick][red](8,2)--(10,2);
	\draw [style = thick][green](12,0)--(12,2);
	\draw [style = thin][black](12,0)--(14,0);
	\draw [style = thick][blue](12,0)--(14,2);

	\draw [style = thick][blue](8,0)--(14,2);
	\draw [style = thick][black](14,0)--(8,2);
	\draw [style = thick][green](10,0)--(12,2);
	\draw [style = thick][red](12,0)--(10,2);
	\draw [style = thick][black] (10,0)..controls(12,-0.5)..(14,0);
	\draw [style = thick][green] (8,2)..controls(10,2.5)..(12,2);
	
	\draw[fill= black] (8,0) circle (.08);
	\draw[fill= white] (8,2) circle (0.2);
	\draw[fill= black] (8,2) circle (0.1) node [above] at (8,2.2) {\small$001$};
	\draw[fill= white] (10,0) circle (0.2);
	\draw[fill= black] (10,0) circle (0.1) node [below] at (10,-0.2) {\small$010$};
	\draw[fill= black] (10,2) circle (.08);
	\draw[fill= white] (12,0) circle (0.2);
	\draw[fill= black] (12,0) circle (0.1) node [left] at (11.8,0) {\small$100$};
	\draw[fill= black] (12,2) circle (.08);
	\draw[fill= black] (14,0) circle (.08);
	\draw[fill= black] (14,2) circle (.08);
	
\draw node [below] at (3,-0.9) {(a). $S = \{000, 001, 011\}$};
\draw node [below] at (11,-0.9
) {(b). $S = \{001, 010, 100\}$};
\draw node [below] at (7,-1.8) {Figure 2. pendant $S$-Steiner trees in $AQ_3$};
\end{tikzpicture}

\end{center}	
In above Figure 2(a), there are three pendant $S$-Steiner trees in $AQ_3$ where $S = \{000, 001, 011\}$ and in Figure 2(b), we get four pendant $S$-Steiner trees in $AQ_3$ with $S = \{001, 010, 100\}.$
		
\noindent
 Recently, Y. Mao \cite{mao1, mao2} worked on pendant tree connectivity.
It is clear that, $ \tau_{k}(G) \leq \kappa_{k}(G),~ k \geq 2.$ For the augmented cube $AQ_n$, we have $\tau_{3}(AQ_n) \leq \kappa_{3}(AQ_n) = 2n - 2.$ 
Let $S$ be a vertex set of a triangle in $AQ_n.$ Here we can not use the two edges from each vertex in $S$ to the other two vertices of $S$ in the construction of $S$-Steiner trees as the vertices of $S$ should be pendant in each tree. Thus, in this case, we can get at most $2n - 3$ pendant $3$-trees. Hence, $\displaystyle \tau_{3}(AQ_n) \leq 2n-3.$ 

In this section, we prove the existence of $2n - 3$ pendant $3$-trees in $AQ_n$ for any choice of subset $S \subset V(AQ_n)$ with $|S| = 3.$ Thus, the result is optimal.

M. Hager\cite{hg1} gave the following result about pendant k-tree-connectivity , $\tau_k(G)$, of a simple, finite graph $G.$
\begin{proposition} [\cite{hg1}] 
	Let $G$ be a graph with $\tau_k(G) \geq m.$ Then $\delta(G) \geq k + m - 1.$
\end{proposition}

We require the above Proposition to prove the next result regarding pendant 3-tree-connectivity of the augmented cube $AQ_n.$
\begin{thm}
	Let $n \geq 3$ be an integer. The pendant $3$-tree-connectivity of $AQ_n,~
	\tau_3(AQ_n)$ is $2n - 3.$ 
\end{thm}
\begin{proof} 
\noindent
The contra-positive statement of the above Proposition 3.3 is:

\noindent
Let $G$ be a graph with $\delta(G) < k + m - 1.$ Then $\tau_k(G) < m.$

\noindent
Now $\delta(AQ_n)= 2n-1 < 3 + 2n - 2 - 1.$ Then from above result, we get $\tau_3(AQ_n) < 2n - 2.$ i.e. $\tau_3(AQ_n) \leq  2n - 3.$

Thus, it is sufficient to show that for any subset $S$ of $V(AQ_n)$ with $|S| = 3,$ there exists $2n - 3$ pendant $S$-Steiner trees in $AQ_n.$

We will prove this result by the induction on $n.$ Let $n = 3.$ Since $AQ_n$ is vertex transitive, from the following figures we get that the result is true for $n = 3.$ 

\begin{center}
	\hspace*{-1cm}

	
	Figure 3. Pendant $S$-Steiner trees in $AQ_3$
	
\end{center}

For $n = 4,$ see the following figures.

\vspace*{-0.1cm}
\hspace*{-2.5cm}
	%
 

\newpage 
Thus, the result is true for $n=3$ and $n=4.$ Suppose by the induction hypothesis, result is true for $AQ_{n-1}$ i.e. $\tau_3(AQ_{n-1}) = 2n - 5.$ 	
Consider the canonical representation of $AQ_n$ as $AQ_n = AQ^0_{n-1} \cup AQ^1_{n-1} \cup E^h_n \cup E^c_n.$
Let $S = \{ x, y, z\}$ be a subset of $V(AQ_n)$ 
such that $|S| = 3.$ \\

\noindent
{\bf Case 1:} Suppose $x, y, z \in V(AQ^0_{n-1}).$ 

By the induction hypothesis, we get $2n - 5$ S-trees in $AQ_{n-1}^0.$ Now in $AQ_{n-1}^1,~ x^h$ is the complement of $x^c.$ As we have the decomposition of $AQ^1_{n-1}$ into two subgraphs namely $AQ^{10}_{n-2}$ and $AQ^{11}_{n-2},$ one of them should lies in $AQ^{10}_{n-2}$ and that the other in $AQ^{11}_{n-2}.$ Similar is true for $\{y^h, y^c\}$ and $\{z^h, z^c\}.$ Thus, one of the neighbour, either hypercubic or complement of each vertex of $S$ lies in $A^{10}_{n-2}$ and other in $AQ^{11}_{n-2}.$ Since $AQ_n$ is Hamiltonian connected, there exists hamiltonian path
$P_1$ in $AQ^{10}_{n-2}$ and $P_2$ in $AQ^{11}_{n-2}.$ Hence joining $x, y$ and $z$ to their neighbours on $P_1$ and on $P_2$ we get two more $S$-Steiner trees. Thus, in this case we get $2n - 3$ $S$-Steiner trees in $AQ_n,$ see Figure 4.\\

\begin{center}
	\vspace*{-0.7cm}
	\begin{tikzpicture}[scale=1.1]
	
	\draw [style = thin](0,0)--(0,6);
	\draw [style = thin](0,0)--(4,0);
	\draw [style = thin](0,6)--(4,6);
	\draw [style = thin](4,0)--(4,6);
	
	\draw [style = thin](6,0)--(6,6);
	\draw [style = thin](6,0)--(10,0);
	\draw [style = thin](6,6)--(10,6);
	\draw [style = thin](10,0)--(10,6);
	\draw [style = thin](6,3)--(10,3);
	
	\draw [fill=black] (1,5) circle  (.08)  node [left]  at (1,4.9) {$x$}; 
	\draw [fill=black] (2,5.5) circle  (.08);
	\draw [fill=black] (2,5) circle  (.08);
	\draw [fill=black] (2,4.5) circle  (.08);
	
	\draw [fill=black] (1,3) circle  (.08)  node [left]  at (1,3) {$y$}; 
	\draw [fill=black] (2,3.5) circle  (.08);
	\draw [fill=black] (2,3) circle  (.08);
	\draw [fill=black] (2,2.5) circle  (.08);
	
	\draw [fill=black] (1,1) circle  (.08)  node [left]  at (1.1,0.7) {$z$}; 
	\draw [fill=black] (2,1.5) circle (.08); 
	\draw [fill=black] (2,1) circle  (.08) ;
	\draw [fill=black] (2,0.5) circle  (.08);
	
	\draw [style = thin](1,5)--(2,5.5);
	\draw [style = thin](1,5)--(2,5);
	\draw [style = thin](1,5)--(2,4.5);
	
	\draw [style = thin](1,3) --(2,3.5);
	\draw [style = thin](1,3) --(2,3);
	\draw [style = thin](1,3) --(2,2.5);
	
	\draw [style = thin](1,1)--(2,1.5);
	\draw [style = thin](1,1)--(2,1);
	\draw [style = thin](1,1)--(2,0.5);
	
	\draw [style = thin] (2,5.5)..controls(3, 4.5)..(2,3.5);
	\draw [style = thin] (2,3.5)..controls(3, 2.5)..(2,1.5);
	
	\draw [style = thin] (2,5)..controls(3, 4)..(2,3);
	\draw [style = thin] (2,3)..controls(3, 2)..(2,1);
	
	\draw [style = thin] (2,4.5)..controls(3, 3.5)..(2,2.5);
	\draw [style = thin] (2,2.5)..controls(3, 1.5)..(2,0.5);
	
	\draw [fill=black] (7,5.5) circle  (.08)  node [right]  at (7,5.5) {$x^h$};
	\draw [fill=black] (8,4.5) circle  (.08)  node [above]  at (8,4.5) {$y^c$};
	\draw [fill=black] (9,3.5) circle  (.08)  node [right]  at (9,3.5)(1,3) {$z^h$};
	\draw [fill=black] (7,2.5) circle  (.08)  node [right]  at (7,2.5) {$z^c$};
	\draw [fill=black] (8,1.5) circle  (.08)  node [above]  at (8,1.5) {$y^h$};
	\draw [fill=black] (9,0.5) circle  (.08)  node [right]  at (9,0.5) {$x^c$};

	\draw [style = thick] [blue] plot [smooth, tension=1.5] coordinates { (7,5.5) (6.5, 3.5) (8,4.5) (9,3.5)};
	\draw [style = thick] [red] plot [smooth, tension=1.5] coordinates { (7,2.5) (6.5, 0.5) (8,1.5) (9,0.5)};
	
	\draw [style = thick] [blue] plot [smooth, tension=1.5] coordinates { (1,5) (3, 6.5) (7,5.5)};
	\draw [style = thick] [red] plot [smooth, tension=1.5] coordinates { (1,5) (1.9, 4) (9,0.5)};
	\draw [style = thick] [red] plot [smooth, tension=1.5] coordinates {  (1,3) (3, 3.7) (8,1.5)};
	\draw [style = thick] [blue] plot [smooth, tension=1.5] coordinates {  (1,3) (1.9, 2) (8,4.5)};
	\draw [style = thick] [blue] plot [smooth, tension=1.5] coordinates {  (1,1) (4, 0.2) (9,3.5)};
	\draw [style = thick] [red] plot [smooth, tension=1.5] coordinates {  (1,1) (1.9, 1.7) (7,2.5)};

	\node [below] at (2,-0.5) {$AQ^0_{n-1}$};
	\node [below] at (8,-0.5) {$AQ^1_{n-1}$};
	\node [below] at (5,-1) {Figure 4.};
	\end{tikzpicture}
	
	
\end{center}

\noindent
{\bf Case 2:} Suppose $\{x, y\} \subseteq V(AQ^0_{n-1})$ and $ z \in V(AQ^1_{n-1}).$

\noindent
{\bf Subcase 2.1:} Let $z \in \{x^h, x^c, y^h, y^c\}.$ 

Without loss of generality, suppose $z = x^h.$

\noindent
{\bf Subcase 2.1.1 :} Suppose $\{x^h , x^c\} = \{y^h , y^c\}.$ 

Observe that $x$ is adjacent to $y.$
We have a path cover $P_1, P_2, \dots, P_{2n-3}$ in between $x$ and $y$ in $AQ_{n-1}^0.$ Let $y_1, y_2, \dots, y_{2n-3}$ be the neighbours of $y$ along $P_1, P_2, \dots, P_{2n-3}$ respectively. Without loss of generality, suppose $y_1 = x.$ Then $y_1^c= x^c= y^h \in AQ_{n - 1}^1.$ Let $Q_1, Q_2, \dots, Q_{2n - 3}$ be the corresponding path cover in between $z(=x^h)$ and $y^h.$ Clearly, the neighbours $y^c_1, y^c_2, \dots, y^c_{2n -3} $ of $y^c$ lie on $Q_1, Q_2, \dots, Q_{2n - 3}$ respectively and hence $Q_1 = <y^c~,~ x^c>.$ Thus, the required $2n - 3$ pendant $S$-Steiner trees $T_1, T_2, \dots, T_{2n -3}$ are as follows:\\

\noindent
$T_i = P_i \cup \{ <y^c~,~ y_i^c>, <y_i~ ,~ y^c_i>\},$ for $2 \leq i \leq 2n -3$ and

\noindent
$T_1 = <x ~,~ x^c=y^h> \cup <y~,~ y^h=x^c > \cup\; Q_1,$ see Figure 5.

\begin{center}
	\begin{tikzpicture}[scale=1.1]
	
	\draw [style = thin](0,0)--(0,6);
	\draw [style = thin](0,0)--(4,0);
	\draw [style = thin](0,6)--(4,6);
	\draw [style = thin](4,0)--(4,6);
	
	\draw [style = thin](6,0)--(6,6);
	\draw [style = thin](6,0)--(10,0);
	\draw [style = thin](6,6)--(10,6);
	\draw [style = thin](10,0)--(10,6);
	
	\draw [style = thick] [blue] plot [smooth, tension=1.5] coordinates {  (1,5) (4.5, 4) (7,1) };
	\draw [style = thick] [blue] plot [smooth, tension=1.5]  coordinates {   (1,1) (4.5, 0.5) (7,1) };
	\draw [style = thick] [blue] plot [smooth, tension=1.5] coordinates { (7,5) (7,1) };	
	
	\draw [fill=black] (1,1) circle  (.08)  node [right]  at (0.8,0.6) {{$y$}};
	\draw [fill=black] (1,5) circle  (.08)  node [above]  at (1,5) {{$x=y_1$}}; 
	\draw [fill=black] (1.5,2) circle  (.08) node [left]  at (1.5,2) {{$y_2$}};
	\draw [fill=black] (2,2) circle  (.08) ;
	\draw [fill=black] (2.5,2) circle  (.08);
	\draw [fill=black] (3,2) circle  (.08) node [right]  at (3,2) {{$y_{2n-3}$}};

	\draw  [style = dotted] plot [smooth, tension=1.5] coordinates { (1,1) (1,5) };
	\draw  plot [smooth, tension=1.5] coordinates { (1,1) (1.5,2) };
	\draw  plot [smooth, tension=1.5] coordinates { (1,1) (2,2) };
	\draw  plot [smooth, tension=1.5] coordinates { (1,1) (2.5,2) };
	\draw  plot [smooth, tension=1.5] coordinates { (1,1) (3,2) };

	\draw plot [smooth, tension=1.5] coordinates { (1,5) (1.7 , 3.5) (1.5,2)};
	\draw plot [smooth, tension=1.5] coordinates {  (1,5) (2 , 3.5) (2,2)};
	\draw plot [smooth, tension=1.5] coordinates {   (1,5) (2.2 , 3.5) (2.5,2)};
	\draw plot [smooth, tension=1.5] coordinates {   (1,5) (2.5 , 3.5) (3,2)};


	\draw [fill=black] (7.5,4) circle  (.08) node [above]  at (7.6,3.9) {{$y_2^c$}};	
	\draw [fill=black] (7,1) circle  (.08)  node [right]  at (7,1) {{$y^h($=$x^c)$}};
	\draw [fill=black] (8,4) circle  (.08);
	\draw [fill=black] (7,5) circle  (.08)  node [right]  at (7,5.1) {{$z$=$x^h($=$y^c)$}};
	\draw [fill=black] (8.5,4) circle  (.08);
	\draw [fill=black] (9,4) circle  (.08) node [right]  at (9,4.1) {{$y_{2n-3}^c$}};

	\draw plot [smooth, tension=1.5] coordinates { (7,5) (8.5,4) };
	\draw plot [smooth, tension=1.5] coordinates { (7,5) (7.5,4) };
	\draw plot [smooth, tension=1.5] coordinates {  (7,5) (9,4) };
	\draw plot [smooth, tension=1.5] coordinates {  (7,5) (8,4) };
	
	\draw [style = dotted] plot [smooth, tension=1.5] coordinates { (8.5,4) (8.4,2.6) (7,1)};
	\draw [style = dotted] plot [smooth, tension=1.5] coordinates { (7.5,4) (7.5,2.6) (7,1)};
	\draw [style = dotted] plot [smooth, tension=1.5] coordinates { (9,4) (8.8,2.6) (7,1)};
	\draw [style = dotted] plot [smooth, tension=1.5] coordinates { (8,4) (8,2.6) (7,1)};


	\draw  plot [smooth, tension=1.5] coordinates { (2.5,2) (6, 2) (8.5,4) };
	\draw  plot [smooth, tension=1.5] coordinates {  (1.5,2) (5, 2) (7.5,4) };
	\draw  plot [smooth, tension=1.5] coordinates {   (3,2) (6.5, 2) (9,4) };
	\draw  plot [smooth, tension=1.5] coordinates {  (2,2) (5.5, 2) (8,4) };
	
	\end{tikzpicture}
	
	\hspace{0.2 in} {$AQ_{n-1}^0$} \hspace{1.7 in} {$AQ_{n-1}^1$}
	
	\hspace{0.3 in} {$ \mathrm{Figure}\,\, 5.$}
	
\end{center}

{\bf Subcase 2.1.2:}
Suppose $\{x^h, x^c\} \neq \{y^h, y^c\}$ and $z$ is not adjacent to $y^h$ and $y^c$ both. Hence, $x$ is not adjacent to $y.$ We know that $AQ_n$ has one-to-one path cover of order k,  $1 \leq k \leq 2n - 1$ in between any pair of vertices. Therefore, we get a path cover $P_1, P_2, \dots, P_{2n-3}$ in between $x$ and $y$ in $AQ_{n-1}^0.$ Let $y_1, y_2, \dots, y_{2n-3}$ be the neighbours of $y$ along $P_1, P_2, \dots, P_{2n-3}$ respectively. In $AQ_{n-1}^1,$ let $Q_1, Q_2, \dots, Q_{2n - 3}$ be the corresponding path cover in between $z= x^h$  and $y^h$ such that $y^h_1, y^h_2, \dots, y^h_{2n -3} $ are neighbors of $y^h$ along $Q_1, Q_2, \dots, Q_{2n - 3}$ respectively. Thus the required $2n - 3$ pendant $S$-Steiner trees $T_1, T_2, \dots, T_{2n -3}$ are obtained as follows:\\

\noindent
$T_i = P_i \cup \{Q_i\backslash <y^h~,~ y_i^h>\} ~\cup~ <y_i~ ,~ y^h_i>,$ for $1 \leq i \leq 2n -3,$

\noindent
see Figure 6.\\

\begin{center}
\begin{tikzpicture}[scale=0.9]

\draw [style = thin](0,0)--(0,6);
\draw [style = thin](0,0)--(4,0);
\draw [style = thin](0,6)--(4,6);
\draw [style = thin](4,0)--(4,6);

\draw [style = thin](6,0)--(6,6);
\draw [style = thin](6,0)--(10,0);
\draw [style = thin](6,6)--(10,6);
\draw [style = thin](10,0)--(10,6);

\draw [fill=black] (2,5) circle  (.08)  node [above]  at (2,5) {{$x$}};

\draw [fill=black] (2,1) circle  (.08)  node [below]  at (2,1) {{$y$}}; 
\draw [fill=black] (2.5,2) circle  (.08);
\draw [fill=black] (1.5,2) circle  (.08);
\draw [fill=black] (1,2) circle  (.08) node [left]  at (1,2) {{$y_1$}};
\draw [fill=black] (3,2) circle  (.08) node [right]  at (3,2) {{$y_{2n-3}$}};
\draw [fill=black] (2,2) circle  (.08);

\draw plot [smooth, tension=1.5] coordinates { (2,1) (2.5,2) };
\draw plot [smooth, tension=1.5] coordinates {  (2,1) (1.5,2) };
\draw plot [smooth, tension=1.5] coordinates {   (2,1) (1,2) };
\draw plot [smooth, tension=1.5] coordinates {   (2,1) (3,2) };
\draw plot [smooth, tension=1.5] coordinates {   (2,1) (2,2) };

\draw plot [smooth, tension=1.5] coordinates { (2.5,2) (2.5,3.5) (2,5)};
\draw plot [smooth, tension=1.5] coordinates {  (1.5,2) (1.5,3.5) (2,5)};
\draw plot [smooth, tension=1.5] coordinates {   (1,2) (1,3.5) (2,5)};
\draw plot [smooth, tension=1.5] coordinates {   (3,2) (3,3.5) (2,5)};
\draw plot [smooth, tension=1.5] coordinates {   (2,2) (2,3.5) (2,5)};


\draw [fill=black] (8,5) circle  (.08)  node [above]  at (8,5) {{$z$}}; 
\draw [fill=black] (8.5,2) circle  (.08);
\draw [fill=black] (7.5,2) circle  (.08);
\draw [fill=black] (7,2) circle  (.08) node [left]  at (7,2) {{$y_1^h$}};
\draw [fill=black] (9,2) circle  (.08) node [right]  at (9,2) {{$y_{2n-3}^h$}};
\draw [fill=black] (8,2) circle  (.08);

\draw [fill=black] (8,1) circle  (.08)  node [below]  at (8,1) {{$y^h$}};

\draw [style = dotted] plot [smooth, tension=1.5] coordinates { (8,1) (8.5,2) };
\draw [style = dotted] plot [smooth, tension=1.5] coordinates {  (8,1) (7.5,2) };
\draw [style = dotted] plot [smooth, tension=1.5] coordinates {   (8,1) (7,2) };
\draw [style = dotted] plot [smooth, tension=1.5] coordinates {   (8,1) (9,2) };
\draw [style = dotted] plot [smooth, tension=1.5] coordinates {  (8,1) (8,2) };

\draw plot [smooth, tension=1.5] coordinates { (8.5,2) (8.5 , 3.5) (8,5)};
\draw plot [smooth, tension=1.5] coordinates {  (7.5,2) (7.5 , 3.5) (8,5)};
\draw plot [smooth, tension=1.5] coordinates {   (7,2) (7 , 3.5) (8,5)};
\draw plot [smooth, tension=1.5] coordinates {   (9,2) (9 , 3.5) (8,5)};
\draw plot [smooth, tension=1.5] coordinates {   (8,2) (8 , 3.5) (8,5)};

\draw  plot [smooth, tension=1.5] coordinates { (2.5,2) (5, 3
	) (8.5,2) };
\draw  plot [smooth, tension=1.5] coordinates {  (1.5,2) (5, 3) (7.5,2) };
\draw  plot [smooth, tension=1.5] coordinates {   (1,2) (5, 3) (7,2) };
\draw  plot [smooth, tension=1.5] coordinates {   (3,2) (5, 3) (9,2) };
\draw  plot [smooth, tension=1.5] coordinates {  (2,2) (5, 3) (8,2) };

\end{tikzpicture}

\hspace{-0.1 in} {$AQ_{n-1}^0$} \hspace{1.5 in} {$AQ_{n-1}^1$}

\hspace{0.3 in} {$ \mathrm{Figure}\,\, 6.$}
\end{center}

{\bf Subcase 2.1.3:} Suppose $\{x^h, x^c\} \neq \{y^h, y^c\}$ and $z$ is adjacent to $y^h$ or both $y^h$ and $y^c$ in $AQ_{n-1}^1.$ Then $y$ is adjacent to $z^h= x$ in $AQ_{n-1}^0.$ We know that $AQ_n$ has one-to-one path cover of order k, $1 \leq k \leq 2n - 1$ in between any pair of vertices. Thus, we get a path cover $P_1, P_2, \dots, P_{2n-3}$ in between $x$ and $y$ in $AQ_{n-1}^0.$ Let $x_1, x_2, \dots, x_{2n-3}$ be the neighbours of $x$ along $P_1, P_2, \dots, P_{2n-3}$ respectively. Similarly, in $AQ_{n-1}^1,$ we get a path cover $Q_1, Q_2, \dots, Q_{2n - 3}$ in between $x^c$ and $z$ such that neighbors $x^c_1, x^c_2, \dots, x^c_{2n -3} $ of $x^c$ lie on $Q_1, Q_2, \dots, Q_{2n - 3}$ respectively. Since $x$ is adjacent to $y$ in $AQ^0_{n-1},$ without loss of generality, we assume that $x_1 = y$ which gives $x_1^c = y^c$ in $AQ^1_{n-1}.$ Also without loss of generality, assume that $x_2 = x^{ch}$ in $AQ^0_{n-1}.$ Hence, $P_1 = \,\,<x~,~y>$ and $Q_2 =\,\, <x^c~,~x^h>.$ The required $2n - 3$ pendant $S$-Steiner trees $T_1, T_2, \dots, T_{2n -3}$ are as follows:\\

\noindent
$T_i = P_i \cup \{Q_i\backslash <x^c~,~ x_i^c>\}~ \cup <x_i~ ,~ x^c_i>,$ for $3 \leq i \leq 2n -3,$

\noindent
$T_1 = P_2 ~\cup <x_2 ~,~ z(=x^h)>$ and
$T_2 =~ <x~,~ x^c> \cup <y~, ~x_1^c (= y^c)> \cup~ Q_1,$ \\ see Figure 7.

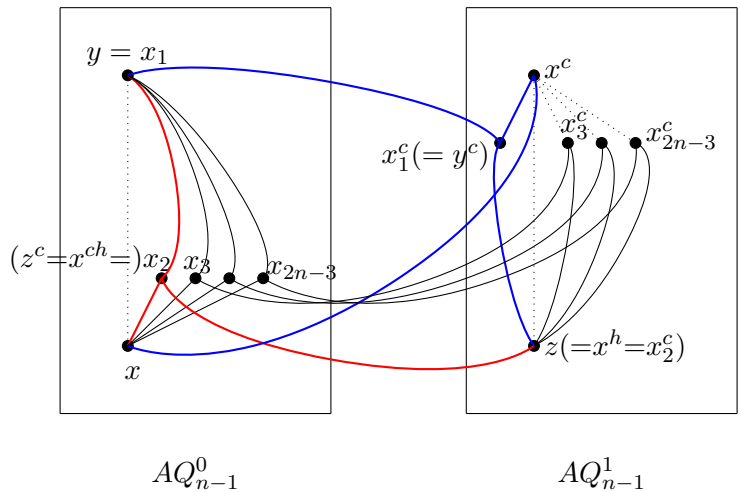
\begin{center}
	\begin{tikzpicture}[scale=0.9]
	
	\draw [style = thin](0,0)--(0,6);
	\draw [style = thin](0,0)--(4,0);
	\draw [style = thin](0,6)--(4,6);
	\draw [style = thin](4,0)--(4,6);
	
	\draw [style = thin](6,0)--(6,6);
	\draw [style = thin](6,0)--(10,0);
	\draw [style = thin](6,6)--(10,6);
	\draw [style = thin](10,0)--(10,6);
	%
	
	\draw [fill=black] (1,1) circle  (.08)  node [right]  at (0.8,0.6) {{$x$}};
	\draw [fill=black] (1,5) circle  (.08)  node [above]  at (1,5) {{$y=x_1$}}; 
	\draw [fill=black] (1.5,2) circle  (.08) node [left]  at (1.8,2.3) {{$(z^c$=$x^{ch}$=)$x_2$}};
	\draw [fill=black] (2,2) circle  (.08) node [left]  at (2.4,2.2) {{$x_3$}};
	\draw [fill=black] (2.5,2) circle  (.08);
	\draw [fill=black] (3,2) circle  (.08) node [right]  at (2.9,2.1) {{$x_{2n-3}$}};

	\draw  [style = dotted] plot [smooth, tension=1.5] coordinates { (1,1) (1,5) };
	\draw  [style = thick] [red] plot [smooth, tension=1.5] coordinates { (1,1) (1.5,2) };
	\draw  plot [smooth, tension=1.5] coordinates { (1,1) (2,2) };
	\draw  plot [smooth, tension=1.5] coordinates { (1,1) (2.5,2) };
	\draw  plot [smooth, tension=1.5] coordinates { (1,1) (3,2) };

	\draw [style = thick] [red] plot [smooth, tension=1.5] coordinates { (1,5) (1.7 , 3.5) (1.5,2)};
	\draw plot [smooth, tension=1.5] coordinates {  (1,5) (2 , 3.5) (2,2)};
	\draw plot [smooth, tension=1.5] coordinates {   (1,5) (2.2 , 3.5) (2.5,2)};
	\draw plot [smooth, tension=1.5] coordinates {   (1,5) (2.5 , 3.5) (3,2)};

	
	\draw [fill=black] (7.5,4) circle  (.08) node [above]  at (7.6,3.9) {{$x_3^c$}};	
	\draw [fill=black] (7,1) circle  (.08)  node [right]  at (7,1) {{$z($=$x^h$=$x_2^c)$}};
	\draw [fill=black] (8,4) circle  (.08);
	\draw [fill=black] (7,5) circle  (.08)  node [right]  at (7,5.1) {{$x^c$}};
	\draw [fill=black] (8.5,4) circle  (.08) node [right]  at (8.5,4.1) {{$x_{2n - 3}^c$}};
	\draw [fill=black] (6.5,4) circle  (.08) node [left]  at (6.5,3.8) {{$x_1^c(=y^c)$}};

	\draw [style = dotted] plot [smooth, tension=1.5] coordinates { (7,5) (8.5,4) };
	\draw [style = dotted] plot [smooth, tension=1.5] coordinates { (7,5) (7.5,4) };
	\draw [style = dotted]  plot [smooth, tension=1.5] coordinates { (7,5) (7,1) };
	\draw [style = thick][blue] plot [smooth, tension=1.5] coordinates {  (7,5) (6.5,4) };
	\draw [style = dotted] plot [smooth, tension=1.5] coordinates {  (7,5) (8,4) };
	
	\draw plot [smooth, tension=1.5] coordinates { (8.5,4) (8.4,2.6) (7,1)};
	\draw plot [smooth, tension=1.5] coordinates {  (7.5,4) (7.5,2.6) (7,1)};
	\draw [style = thick][blue]  plot [smooth, tension=1.5] coordinates {   (6.5,4) (6.5,2.6) (7,1)};
	\draw plot [smooth, tension=1.5] coordinates {   (8,4) (8,2.6) (7,1)};


	\draw  plot [smooth, tension=1.5] coordinates { (2.5,2) (6, 2) (8,4) };
	\draw  [style = thick] [red] plot [smooth, tension=1.5] coordinates {  (1.5,2) (4, 0.8) (7,1) };
	\draw [style = thick] [blue] plot [smooth, tension=1.5]  coordinates {   (1,1) (5, 2) (7,5) };
	\draw  plot [smooth, tension=1.5] coordinates {   (3,2) (6.5, 2) (8.5,4) };
	\draw  plot [smooth, tension=1.5] coordinates {  (2,2) (5.5, 2) (7.5,4) };
	\draw [style = thick] [blue] plot [smooth, tension=1.5] coordinates {  (1,5) (4, 4.9) (6.5, 4)};
	
	\draw node [below] at (2,-0.5) {$AQ_{n-1}^0$};
	\draw node [below] at (8,-0.5) {$AQ_{n-1}^1$};
	
	\end{tikzpicture}
	

	\hspace{0.3 in} {$ \mathrm{Figure}\,\, 7.$}
	
\end{center}

\noindent
Similarly, we get $2n - 3$ pendant $S$-Steiner trees in the augmented cube $AQ_n$ if $z = x^c, y^c$ or $y^h.$\\

\noindent
{\bf Subcase 2.2 :} Let $z \notin \{x^c, x^h, y^c, y^h\}.$\\

\noindent
{\bf Subcase 2.2.1 :} Consider, $\{x^h , x^c\} = \{y^h , y^c\}.$ Observe that in this case, $x$ is adjacent to $y$ and $x^h = y^c, x^c= y^h.$\\

\noindent
{\bf Subcase 2.2.1(a) :} Suppose $z$ is adjacent to one of the vertex $x^h$ or $x^c$. Without loss of generality, suppose $z$ is adjacent to $x^c.$ Now we have a path cover $P_1, P_2, \dots, P_{2n-3}$ in between $x$ and $y$ in $AQ_{n-1}^0.$ Let $y_1, y_2, \dots, y_{2n-3}$ be the neighbours of $y$ along $P_1, P_2, \dots, P_{2n-3}$ respectively. Without loss of generality, suppose $y_1 = x.$ Thus, $P_1 =\,\,<x~,~ y>.$ Then $y_1^c= x^c=y^h \in AQ_{n - 1}^1.$ Consider a path cover $Q_1, Q_2, \dots, Q_{2n - 3}$ in between $y^c(= x^h)$ and $z$ such that neighbours $y^c_1, y^c_2, \dots, y^c_{2n -3} $ of $y^c$ lie on $Q_1, Q_2, \dots, Q_{2n - 3}$ respectively. Thus the required $2n - 3$ pendant $S$-Steiner trees $T_1, T_2, \dots, T_{2n -3}$ are as follows:\\

\noindent
$T_i = P_i \cup \{Q_i\backslash <y^c~,~ y_i^c>\}~ \cup <y_i~ ,~ y^c_i>,$ for $2 \leq i \leq 2n -3.$

\noindent
$T_1 =\,\, <x~ ,~ x^c=y^h> \cup <y ~,~ y^h=x^c > \cup ~\{Q_1 \backslash <x^c ~, ~x^h>\},$\\ see Figure 8.\\

\begin{center}
	\begin{tikzpicture}[scale= 1]
	
	\draw [style = thin](0,0)--(0,6);
	\draw [style = thin](0,0)--(4,0);
	\draw [style = thin](0,6)--(4,6);
	\draw [style = thin](4,0)--(4,6);
	
	\draw [style = thin](6,0)--(6,6);
	\draw [style = thin](6,0)--(10,0);
	\draw [style = thin](6,6)--(10,6);
	\draw [style = thin](10,0)--(10,6);
	%
	
	\draw [fill=black] (1,1) circle  (.08)  node [right]  at (0.8,0.6) {{$y$}};
	\draw [fill=black] (1,5) circle  (.08)  node [above]  at (1,5) {{$x=y_1$}}; 
	\draw [fill=black] (1.5,2) circle  (.08) node [left]  at (1.5,2) {{$y_2$}};
	\draw [fill=black] (2,2) circle  (.08) ;
	\draw [fill=black] (2.5,2) circle  (.08);
	\draw [fill=black] (3,2) circle  (.08) node [right]  at (3,2) {{$y_{2n-3}$}};

	\draw  [style = dotted] plot [smooth, tension=1.5] coordinates { (1,1) (1,5) };
	\draw  plot [smooth, tension=1.5] coordinates { (1,1) (1.5,2) };
	\draw  plot [smooth, tension=1.5] coordinates { (1,1) (2,2) };
	\draw  plot [smooth, tension=1.5] coordinates { (1,1) (2.5,2) };
	\draw  plot [smooth, tension=1.5] coordinates { (1,1) (3,2) };

	\draw plot [smooth, tension=1.5] coordinates { (1,5) (1.7 , 3.5) (1.5,2)};
	\draw plot [smooth, tension=1.5] coordinates {  (1,5) (2 , 3.5) (2,2)};
	\draw plot [smooth, tension=1.5] coordinates {   (1,5) (2.2 , 3.5) (2.5,2)};
	\draw plot [smooth, tension=1.5] coordinates {   (1,5) (2.5 , 3.5) (3,2)};

	
	\draw [fill=black] (8,5) circle  (.08)  node [above]  at (8,5) {$y^c(=x^h)$}; 
	\draw [fill=black] (8.5,4) circle  (.08);
	\draw [fill=black] (7.5,4) circle  (.08);
	\draw [fill=black] (6.5,3.5) circle  (.08) node [above]  at (6.4,3.3) {{$x^c (= y^h )$}};
	\draw [fill=black] (9,4) circle  (.08) node [right]  at (9,4) {{$y_{2n-3}^c$}};
	\draw [fill=black] (8,4) circle  (.08);
	
	\draw [fill=black] (8,1) circle  (.08)  node [below]  at (8,1) {$z$};
	
	\draw [style = dotted] plot [smooth, tension=1.5] coordinates { (8,5) (8.5,4) };
	\draw [style = dotted] plot [smooth, tension=1.5] coordinates {  (8,5) (7.5,4) };
	\draw  [style = dotted] plot [smooth, tension=1.5] coordinates {   (8,5) (6.5,3.5)};
	\draw [style = dotted] plot [smooth, tension=1.5] coordinates {   (8,5) (9,4) };
	\draw [style = dotted] plot [smooth, tension=1.5] coordinates {  (8,5) (8,4) };
	
	\draw plot [smooth, tension=1.5] coordinates { (8.5,4) (8.5 , 2.5) (8,1)};
	\draw plot [smooth, tension=1.5] coordinates {  (7.5,4) (7.5 , 2.5) (8,1)};
	\draw [style = thick] [blue] plot [smooth, tension=1.5] coordinates { (6.5,3.5) (8,1)};
	\draw plot [smooth, tension=1.5] coordinates {   (9,4) (9 , 2.5) (8,1)};
	\draw plot [smooth, tension=1.5] coordinates {   (8,4) (8 , 2.5) (8,1)};


	\draw  plot [smooth, tension=1.5] coordinates { (2.5,2) (6, 2) (8.5,4) };
	\draw  plot [smooth, tension=1.5] coordinates {  (1.5,2) (5, 2) (7.5,4) };
	\draw [style = thick] [blue] plot [smooth, tension=1.5]  coordinates {   (1,1) (5, 1.5) (6.5,3.5) };
	\draw  plot [smooth, tension=1.5] coordinates {   (3,2) (6.5, 2) (9,4) };
	\draw  plot [smooth, tension=1.5] coordinates {  (2,2) (5.5, 2) (8,4) };
	\draw [style = thick] [blue] plot [smooth, tension=1.5] coordinates {  (1,5) (4.5, 4.5) (6.5,3.5) };
	
	\draw node [below] at (2,-0.5) {$AQ_{n-1}^0$};
	\draw node [below] at (8,-0.5) {$AQ_{n-1}^1$};
	
	\end{tikzpicture}
	
	
	\hspace{0.3 in} {$ \mathrm{Figure}\,\, 8.$}
	
\end{center}  

{\bf Subcase 2.2.1(b) :} Suppose $z$ is adjacent to both $x^h$ and $x^c$. Since $z$ is adjacent to $x^h = y^c$ in $AQ_{n-1}^1, \,\, z^c$ is adjacent to $y$ in $AQ_{n-1}^0.$ Now we have a path cover $P_1, P_2, \dots, P_{2n-3}$ in between $x$ and $y$ in $AQ_{n-1}^0.$ Let $y_1, y_2, \dots, y_{2n-3}$ be the neighbours of $y$ along $P_1, P_2, \dots, P_{2n-3}$ respectively. Without loss of generality, suppose $y_1 = z^c$ and $y_2 = x$ Thus, $P_2 = <x~,~y>.$ Then $y_1^c= z$ and $y_2^c = x^c$ in $AQ_{n - 1}^1.$ Consider a path cover $Q_1, Q_2, \dots, Q_{2n - 3}$ in between $y^c(= x^h)$ and $z$ such that the neighbours $y^c_1, y^c_2, \dots, y^c_{2n -3} $ of $y^c$ lie on $Q_1, Q_2, \dots, Q_{2n - 3}$ respectively. The required $2n - 3$ pendant $S$-Steiner trees $T_1, T_2, \dots, T_{2n -3}$ are as follows:\\

\noindent
$T_i = P_i ~\cup \{Q_i\backslash <y^c~,~ y_i^c>\}~ \cup <y_i~ ,~ y^c_i>,$ for $3 \leq i \leq 2n -3.$

\noindent
$T_1 = P_1~ \cup <z^c~ ,~ z> $ and $T_2 = \,\,<x ~,~ x^h> \cup <y ~,~ y^h> \cup~ Q_2,$\\ see Figure 9. \\

\begin{center}
	\begin{tikzpicture}[scale= 1]
	
	\draw [style = thin](0,0)--(0,6);
	\draw [style = thin](0,0)--(4,0);
	\draw [style = thin](0,6)--(4,6);
	\draw [style = thin](4,0)--(4,6);
	
	\draw [style = thin](6,0)--(6,6);
	\draw [style = thin](6,0)--(10,0);
	\draw [style = thin](6,6)--(10,6);
	\draw [style = thin](10,0)--(10,6);

	\draw [style = thick] [blue] plot [smooth, tension=1.5] coordinates {  (1,5) (4.5, 4) (7,5) };
	\draw [style = thick] [red] plot [smooth, tension=1.5]  coordinates {   (0.5,2) (4.2, 2.5) (7,1) };
	\draw [style = thick] [blue] plot [smooth, tension=1.5]  coordinates {   (1,1) (4.5, 1) (7.5 , 3) };	
	

	\draw  [style = dotted] plot [smooth, tension=1.5] coordinates { (1,1) (1,5) };
	\draw  [style = thick][red] plot [smooth, tension=1.5] coordinates { (1,1) (0.5,2) };
	\draw  plot [smooth, tension=1.5] coordinates { (1,1) (1.5,2) };
	\draw  plot [smooth, tension=1.5] coordinates { (1,1) (2,2) };
	\draw  plot [smooth, tension=1.5] coordinates { (1,1) (2.5,2) };
	\draw  plot [smooth, tension=1.5] coordinates { (1,1) (3,2) };

	\draw [style = thick][red] plot [smooth, tension=1.5] coordinates { (1,5) (0.4 , 3.5) (0.5,2)};
	\draw plot [smooth, tension=1.5] coordinates { (1,5) (1.7 , 3.5) (1.5,2)};
	\draw plot [smooth, tension=1.5] coordinates {  (1,5) (2 , 3.5) (2,2)};
	\draw plot [smooth, tension=1.5] coordinates {   (1,5) (2.2 , 3.5) (2.5,2)};
	\draw plot [smooth, tension=1.5] coordinates {   (1,5) (2.5 , 3.5) (3,2)};
	
	\draw [fill=black] (1,1) circle  (.08)  node [right]  at (0.8,0.6) {{$y$}};
	\draw [fill=black] (0.5,2) circle  (.08) node [left]  at (0.6,2) {{$z^c=y_1$}};
	\draw [fill=black] (1,5) circle  (.08)  node [above]  at (1,5) {{$x=y_2$}}; 
	\draw [fill=black] (1.5,2) circle  (.08) node [left]  at (1.6,2) {{$y_3$}};
	\draw [fill=black] (2,2) circle  (.08) ;
	\draw [fill=black] (2.5,2) circle  (.08);
	\draw [fill=black] (3,2) circle  (.08) node [right]  at (3,2.1) {{$y_{2n-3}$}};


	\draw [style = dotted] plot [smooth, tension=1.5] coordinates { (7,5) (8.5,4) };
	\draw [style = thick] [blue] plot [smooth, tension=1.5] coordinates { (7,5) (7.5,3) };
	\draw [style = dotted] plot [smooth, tension=1.5] coordinates { (7,5) (7,1) };
	\draw [style = dotted] plot [smooth, tension=1.5] coordinates {  (7,5) (9,4) };
	\draw [style = dotted] plot [smooth, tension=1.5] coordinates {  (7,5) (8,4) };
	\draw [style = dotted] plot [smooth, tension=1.5] coordinates {  (7,5) (9.5,4) };
	
	\draw plot [smooth, tension=1.5] coordinates { (8.5,4) (8.3,2.2) (7,1)};
	\draw [style = thick] [blue] plot [smooth, tension=1.5] coordinates {  (7.5,3) (7,1)};
	\draw plot [smooth, tension=1.5] coordinates {   (9,4) (8.6,2.1) (7,1)};
	\draw plot [smooth, tension=1.5] coordinates {   (8,4) (8,2.6) (7,1)};
	\draw plot [smooth, tension=1.5] coordinates { (9.5,4) (8.9,2) (7,1)};
	
	\draw [fill=black] (7.5,3) circle  (.08) node [left] at (7.6,3) {{$x^c(=y^h)$}};	
	\draw [fill=black] (7,1) circle  (.08)  node [below]  at (7,1) {{$z$}};
	\draw [fill=black] (8,4) circle  (.08) node [above]  at (8,3.9) {{$y_3^c$}};;
	\draw [fill=black] (7,5) circle  (.08)  node [right]  at (7,5.1) {{$y^c($=$x^h)$}};
	\draw [fill=black] (8.5,4) circle  (.08);
	\draw [fill=black] (9,4) circle (.08);
	\draw [fill=black] (9.5,4) circle (.08) node [right] at (9.5,4.1) {{$y_{2n-3}^c$}};


	\draw  plot [smooth, tension=1.5] coordinates { (2.5,2) (6, 2) (9,4) };
	\draw  plot [smooth, tension=1.5] coordinates {  (1.5,2) (5, 2) (8,4) };
	\draw  plot [smooth, tension=1.5] coordinates {   (3,2) (6.5, 2) (9.5,4) };
	\draw  plot [smooth, tension=1.5] coordinates {  (2,2) (5.5, 2) (8.5,4) };

	\draw node [below] at (2,-0.5) {$AQ_{n-1}^0$};
	\draw node [below] at (8,-0.5) {$AQ_{n-1}^1$};
	
	\end{tikzpicture}
	

	\hspace{0.3 in} {$ \mathrm{Figure}\,\, 9.$}
	
\end{center}

{\bf Subcase 2.2.2 :} Suppose $x$ is not adjacent to $y.$\\

{\bf Subcase 2.2.2(a) :} If $z$ is not adjacent to all $x^h, x^c, y^h,$ and $y^c.$ Then by the same argument in Subcase 2.1.2 of this theorem, we get the required $2n - 3$ pendant $S$-Steiner trees,

see Figure 10.\\

\begin{center}
	\begin{tikzpicture}[scale=1]
	
	\draw [style = thin](0,0)--(0,6);
	\draw [style = thin](0,0)--(4,0);
	\draw [style = thin](0,6)--(4,6);
	\draw [style = thin](4,0)--(4,6);
	
	\draw [style = thin](6,0)--(6,6);
	\draw [style = thin](6,0)--(10,0);
	\draw [style = thin](6,6)--(10,6);
	\draw [style = thin](10,0)--(10,6);

	\draw [fill=black] (2,5) circle  (.08)  node [above]  at (2,5) {{$x$}};

	\draw [fill=black] (2,1) circle  (.08)  node [below]  at (2,1) {{$y$}}; 
	\draw [fill=black] (2.5,2) circle  (.08);
	\draw [fill=black] (1.5,2) circle  (.08);
	\draw [fill=black] (1,2) circle  (.08) node [left]  at (1,2) {{$y_1$}};
	\draw [fill=black] (3,2) circle  (.08) node [right]  at (3,2) {{$y_{2n-3}$}};
	\draw [fill=black] (2,2) circle  (.08);
	
	\draw plot [smooth, tension=1.5] coordinates { (2,1) (2.5,2) };
	\draw plot [smooth, tension=1.5] coordinates {  (2,1) (1.5,2) };
	\draw plot [smooth, tension=1.5] coordinates {   (2,1) (1,2) };
	\draw plot [smooth, tension=1.5] coordinates {   (2,1) (3,2) };
	\draw plot [smooth, tension=1.5] coordinates {   (2,1) (2,2) };
	
	\draw plot [smooth, tension=1.5] coordinates { (2.5,2) (2.5,3.5) (2,5)};
	\draw plot [smooth, tension=1.5] coordinates {  (1.5,2) (1.5,3.5) (2,5)};
	\draw plot [smooth, tension=1.5] coordinates {   (1,2) (1,3.5) (2,5)};
	\draw plot [smooth, tension=1.5] coordinates {   (3,2) (3,3.5) (2,5)};
	\draw plot [smooth, tension=1.5] coordinates {   (2,2) (2,3.5) (2,5)};
	
	
	\draw [fill=black] (8,5) circle  (.08)  node [above]  at (8,5) {{$z$}}; 
	\draw [fill=black] (8.5,2) circle  (.08);
	\draw [fill=black] (7.5,2) circle  (.08);
	\draw [fill=black] (7,2) circle  (.08) node [left]  at (7,2) {{$y_1^h$}};
	\draw [fill=black] (9,2) circle  (.08) node [right]  at (9,2) {{$y_{2n-3}^h$}};
	\draw [fill=black] (8,2) circle  (.08);
	
	\draw [fill=black] (8,1) circle  (.08)  node [below]  at (8,1) {{$y^h$}};
	
	\draw [style = dotted] plot [smooth, tension=1.5] coordinates { (8,1) (8.5,2) };
	\draw [style = dotted] plot [smooth, tension=1.5] coordinates {  (8,1) (7.5,2) };
	\draw [style = dotted] plot [smooth, tension=1.5] coordinates {   (8,1) (7,2) };
	\draw [style = dotted] plot [smooth, tension=1.5] coordinates {   (8,1) (9,2) };
	\draw [style = dotted] plot [smooth, tension=1.5] coordinates {  (8,1) (8,2) };
	
	\draw plot [smooth, tension=1.5] coordinates { (8.5,2) (8.5 , 3.5) (8,5)};
	\draw plot [smooth, tension=1.5] coordinates {  (7.5,2) (7.5 , 3.5) (8,5)};
	\draw plot [smooth, tension=1.5] coordinates {   (7,2) (7 , 3.5) (8,5)};
	\draw plot [smooth, tension=1.5] coordinates {   (9,2) (9 , 3.5) (8,5)};
	\draw plot [smooth, tension=1.5] coordinates {   (8,2) (8 , 3.5) (8,5)};

	\draw  plot [smooth, tension=1.5] coordinates { (2.5,2) (5, 3) (8.5,2) };
	\draw  plot [smooth, tension=1.5] coordinates {  (1.5,2) (5, 3) (7.5,2) };
	\draw  plot [smooth, tension=1.5] coordinates {   (1,2) (5, 3) (7,2) };
	\draw  plot [smooth, tension=1.5] coordinates {   (3,2) (5, 3) (9,2) };
	\draw  plot [smooth, tension=1.5] coordinates {  (2,2) (5, 3) (8,2) };
	
	\draw node [below] at (2,-0.5) {$AQ_{n-1}^0$};
	\draw node [below] at (8,-0.5) {$AQ_{n-1}^1$};

	\end{tikzpicture}
	

	\hspace{0.3 in} {$ \mathrm{Figure}\,\, 10.$}
	
\end{center}

{\bf Subcase 2.2.2(b) :} 
If $z$ is adjacent $y^c$ or $y^h$ or both but not adjacent to $x^h, x^c$. Here, with the similar argument in Subcase 2.1.2 of this theorem, we get the required $2n - 3$ pendant $S$-Steiner trees,\\ see Figure 11.

\begin{center}
		\begin{tikzpicture}[scale=1]
	
	\draw [style = thin](0,0)--(0,6);
	\draw [style = thin](0,0)--(4,0);
	\draw [style = thin](0,6)--(4,6);
	\draw [style = thin](4,0)--(4,6);
	
	\draw [style = thin](6,0)--(6,6);
	\draw [style = thin](6,0)--(10,0);
	\draw [style = thin](6,6)--(10,6);
	\draw [style = thin](10,0)--(10,6);

	\draw [fill=black] (2,5) circle  (.08)  node [above]  at (2,5) {{$y$}};

	\draw [fill=black] (2,1) circle  (.08)  node [below]  at (2,1) {{$x$}}; 
	\draw [fill=black] (2.5,2) circle  (.08);
	\draw [fill=black] (1.5,2) circle  (.08);
	\draw [fill=black] (1,2) circle  (.08) node [left]  at (1,2) {{$x_1$}};
	\draw [fill=black] (3,2) circle  (.08) node [right]  at (3,2) {{$x_{2n-3}$}};
	\draw [fill=black] (2,2) circle  (.08);
	
	\draw plot [smooth, tension=1.5] coordinates { (2,1) (2.5,2) };
	\draw plot [smooth, tension=1.5] coordinates {  (2,1) (1.5,2) };
	\draw plot [smooth, tension=1.5] coordinates {   (2,1) (1,2) };
	\draw plot [smooth, tension=1.5] coordinates {   (2,1) (3,2) };
	\draw plot [smooth, tension=1.5] coordinates {   (2,1) (2,2) };
	
	\draw plot [smooth, tension=1.5] coordinates { (2.5,2) (2.5,3.5) (2,5)};
	\draw plot [smooth, tension=1.5] coordinates {  (1.5,2) (1.5,3.5) (2,5)};
	\draw plot [smooth, tension=1.5] coordinates {   (1,2) (1,3.5) (2,5)};
	\draw plot [smooth, tension=1.5] coordinates {   (3,2) (3,3.5) (2,5)};
	\draw plot [smooth, tension=1.5] coordinates {   (2,2) (2,3.5) (2,5)};
	
	
	\draw [fill=black] (8,5) circle  (.08)  node [above]  at (8,5) {{$z$}}; 
	\draw [fill=black] (8.5,2) circle  (.08);
	\draw [fill=black] (7.5,2) circle  (.08);
	\draw [fill=black] (7,2) circle  (.08) node [left]  at (7,2) {{$x_1^h$}};
	\draw [fill=black] (9,2) circle  (.08) node [right]  at (9,2) {{$x_{2n-3}^h$}};
	\draw [fill=black] (8,2) circle  (.08);
	
	\draw [fill=black] (8,1) circle  (.08)  node [below]  at (8,1) {{$x^h$}};
	
	\draw [style = dotted] plot [smooth, tension=1.5] coordinates { (8,1) (8.5,2) };
	\draw [style = dotted] plot [smooth, tension=1.5] coordinates {  (8,1) (7.5,2) };
	\draw [style = dotted] plot [smooth, tension=1.5] coordinates {   (8,1) (7,2) };
	\draw [style = dotted] plot [smooth, tension=1.5] coordinates {   (8,1) (9,2) };
	\draw [style = dotted] plot [smooth, tension=1.5] coordinates {  (8,1) (8,2) };
	
	\draw plot [smooth, tension=1.5] coordinates { (8.5,2) (8.5 , 3.5) (8,5)};
	\draw plot [smooth, tension=1.5] coordinates {  (7.5,2) (7.5 , 3.5) (8,5)};
	\draw plot [smooth, tension=1.5] coordinates {   (7,2) (7 , 3.5) (8,5)};
	\draw plot [smooth, tension=1.5] coordinates {   (9,2) (9 , 3.5) (8,5)};
	\draw plot [smooth, tension=1.5] coordinates {   (8,2) (8 , 3.5) (8,5)};

	\draw  plot [smooth, tension=1.5] coordinates { (2.5,2) (5, 3
		) (8.5,2) };
	\draw  plot [smooth, tension=1.5] coordinates {  (1.5,2) (5, 3) (7.5,2) };
	\draw  plot [smooth, tension=1.5] coordinates {   (1,2) (5, 3) (7,2) };
	\draw  plot [smooth, tension=1.5] coordinates {   (3,2) (5, 3) (9,2) };
	\draw  plot [smooth, tension=1.5] coordinates {  (2,2) (5, 3) (8,2) };
	
	\draw node [below] at (2,-0.5) {$AQ_{n-1}^0$};
	\draw node [below] at (8,-0.5) {$AQ_{n-1}^1$};	
	
	\end{tikzpicture}
	

	\hspace{0.5 in} {$ \mathrm{Figure}\,\, 11.$}
\end{center}

In the same way, we get $2n - 3$ pendant $S$-Steiner trees if $z$ is adjacent to $x^h$ or $x^c$ or both but not adjacent to $y^h, y^c$.\\ 

\noindent
{\bf Subcase 2.2.2(c) :} If $z$ is adjacent to $y^h$ only or to $y^h, x^c, x^h$ then use a path cover of order $2n -3$ in between $y^c$ instead of $y^h$ and $z$ in $AQ_{n-1}^1$ and get the required result as similar to Subcase 2.2.2(a).\\ Similarly we get $2n - 3$ pendant $S$-Steiner trees if  $z$ is adjacent to $x^h$ only or to $x^h,~ y^c,~ y^h.$\\

{\bf Subcase 2.2.3 :} Assume now that $x$ is adjacent to $y.$\\

{\bf Subcase 2.2.3(a) :} Suppose $z$ is not adjacent to all $x^h, x^c, y^h$ and $y^c.$
We have a path cover $P_1, P_2, \dots, P_{2n-3}$ in between $x$ and $y$ in $AQ_{n-1}^0.$ Let $y_1, y_2, \dots, y_{2n-3}$ be the neighbours of $y$ along $P_1, P_2, \dots, P_{2n-3}$ respectively. Since $x$ is adjacent to $y,$ without loss of generality, suppose $y_1 = x.$ Hence $y_1^h = x^h$. Now, in $AQ_{n-1}^1,$ we get a path cover $Q_1, Q_2, \dots, Q_{2n - 3}$ in between $y^h$ and $z$ such that the neighbours $y^h_1, y^h_2, \dots, y^h_{2n -3} $ of $y^h$ lie on $Q_1, Q_2, \dots, Q_{2n - 3}$ respectively. Thus, the required $2n - 3$ pendant $S$-Steiner trees $T_1, T_2, \dots, T_{2n -3}$ are obtained as follows:\\

$T_i = P_i \cup \{Q_i\backslash <y^h~,~ y_i^h>\}~ \cup <y_i ~,~ y^h_i>,$ for $2 \leq i \leq 2n -3$ and

$T_1 = \,\,<x~,~x^h> \cup <y~,~y^h> \cup~ Q_1,$ see Figure 12.

\begin{center}
	\begin{tikzpicture}[scale=1]
	
	\draw [style = thin](0,0)--(0,6);
	\draw [style = thin](0,0)--(4,0);
	\draw [style = thin](0,6)--(4,6);
	\draw [style = thin](4,0)--(4,6);
	
	\draw [style = thin](6,0)--(6,6);
	\draw [style = thin](6,0)--(10,0);
	\draw [style = thin](6,6)--(10,6);
	\draw [style = thin](10,0)--(10,6);
	%
	
	\draw [fill=black] (1,1) circle  (.08)  node [right]  at (0.8,0.6) {{$y$}};
	\draw [fill=black] (1,5) circle  (.08)  node [above]  at (1,5) {{$x=y_1$}}; 
	\draw [fill=black] (1.5,2) circle  (.08) node [left]  at (1.5,2) {{$y_2$}};
	\draw [fill=black] (2,2) circle  (.08) ;
	\draw [fill=black] (2.5,2) circle  (.08);
	\draw [fill=black] (3,2) circle  (.08) node [right]  at (3,2) {{$y_{2n-3}$}};

	\draw  [style = dotted] plot [smooth, tension=1.5] coordinates { (1,1) (1,5) };
	\draw  plot [smooth, tension=1.5] coordinates { (1,1) (1.5,2) };
	\draw  plot [smooth, tension=1.5] coordinates { (1,1) (2,2) };
	\draw  plot [smooth, tension=1.5] coordinates { (1,1) (2.5,2) };
	\draw  plot [smooth, tension=1.5] coordinates { (1,1) (3,2) };

	\draw plot [smooth, tension=1.5] coordinates { (1,5) (1.7 , 3.5) (1.5,2)};
	\draw plot [smooth, tension=1.5] coordinates {  (1,5) (2 , 3.5) (2,2)};
	\draw plot [smooth, tension=1.5] coordinates {   (1,5) (2.2 , 3.5) (2.5,2)};
	\draw plot [smooth, tension=1.5] coordinates {   (1,5) (2.5 , 3.5) (3,2)};

	
	\draw [fill=black] (7,2) circle (.08) node [left]  at (7,1.7) {{$y_1^h
			($=$x^h)$}};	
	\draw [fill=black] (7.5,2) circle  (.08) node [below]  at (7.7,2.09) {{$y_2^h$}};	
	\draw [fill=black] (8,1) circle  (.08)  node [below]  at (8,1) {{$y^h$}};
	\draw [fill=black] (8,2) circle  (.08);
	\draw [fill=black] (8,5) circle  (.08)  node [above]  at (8,5) {{$z$}};
	\draw [fill=black] (8.5,2) circle  (.08);
	\draw [fill=black] (9,2) circle  (.08) node [right]  at (9,2) {{$y_{2n-3}^h$}};

	\draw [style = dotted] plot [smooth, tension=1.5] coordinates { (8,1) (8.5,2) };
	\draw [style = dotted] plot [smooth, tension=1.5] coordinates { (8,1) (7.5,2) };
	\draw [style = thick] [blue] plot [smooth, tension=1.5] coordinates { (8,1) (7,2) };
	\draw [style = dotted] plot [smooth, tension=1.5] coordinates {  (8,1) (9,2) };
	\draw [style = dotted] plot [smooth, tension=1.5] coordinates {  (8,1) (8,2) };
	
	\draw plot [smooth, tension=1.5] coordinates { (8.5,2) (8.5,3.5) (8,5)};
	\draw plot [smooth, tension=1.5] coordinates {  (7.5,2) (7.5,3.5) (8,5)};
	\draw [style = thick][blue]  plot [smooth, tension=1.5] coordinates {   (7,2) (7,3.5) (8,5)};
	\draw plot [smooth, tension=1.5] coordinates {   (9,2) (9,3.5) (8,5)};
	\draw plot [smooth, tension=1.5] coordinates {   (8,2) (8,3.5) (8,5)};


	\draw  plot [smooth, tension=1.5] coordinates { (2.5,2) (6, 3) (8.5,2) };
	\draw  plot [smooth, tension=1.5] coordinates {  (1.5,2) (5, 3) (7.5,2) };
	\draw [style = thick] [blue] plot [smooth, tension=1.5]  coordinates {   (1,1) (5, 0.5) (8,1) };
	\draw  plot [smooth, tension=1.5] coordinates {   (3,2) (6.5, 3) (9,2) };
	\draw  plot [smooth, tension=1.5] coordinates {  (2,2) (5.5, 3) (8,2) };
	\draw [style = thick] [blue] plot [smooth, tension=1.5] coordinates {  (1,5) (4.5, 3) (7,2) };

	\draw node [below] at (2,-0.5) {$AQ_{n-1}^0$};
	\draw node [below] at (8,-0.5) {$AQ_{n-1}^1$};
	
	\end{tikzpicture}
	
	
	\hspace{0.5 in} {$ \mathrm{Figure}\,\, 12.$}
	
\end{center}

{\bf Subcase 2.2.3(b) :} 
Suppose $z$ is adjacent $y^c$ or $y^h$ or both but not adjacent to $x^h$ and $ x^c$. Here, we will take a path cover $P_1, P_2, \dots, P_{2n-3}$ in between $x$ and $y$ in $AQ_{n-1}^0.$ Let $x_1, x_2, \dots, x_{2n-3}$ be the neighbours of $x$ along $P_1, P_2, \dots, P_{2n-3}$ respectively. Also, in $AQ_{n-1}^1,$ we get a path cover $Q_1, Q_2, \dots, Q_{2n - 3}$ in between $x^h$ and $z$ such that neighbours $x^h_1, x^h_2, \dots, x^h_{2n -3} $ of $x^h$ lie on $Q_1, Q_2, \dots, Q_{2n - 3}$ respectively. Since $x$ is adjacent to $y,$ without loss of generality, suppose $x_1 = y.$ Which gives us $P_1 = < x~ , ~y>$ and $x_1^h = y^h.$ Thus the required $2n - 3$ pendant $S$-Steiner trees $T_1, T_2, \dots, T_{2n -3}$ are constructed as follows:\\

\noindent
$T_i = P_i \cup \{Q_i\backslash <x^h~,~ x_i^h>\}~ \cup <x_i~ ,~ x^h_i>,$ for $2 \leq i \leq 2n -3.$ and 

\noindent
$T_1 =\,\, <x ~, ~ x^h> \cup <y ~, ~ y^h> \cup~ Q_1,$ see Figure 13.\\

\begin{center}
		\begin{tikzpicture}[scale= 1]
	
	\draw [style = thin](0,0)--(0,6);
	\draw [style = thin](0,0)--(4,0);
	\draw [style = thin](0,6)--(4,6);
	\draw [style = thin](4,0)--(4,6);
	
	\draw [style = thin](6,0)--(6,6);
	\draw [style = thin](6,0)--(10,0);
	\draw [style = thin](6,6)--(10,6);
	\draw [style = thin](10,0)--(10,6);
	%
	
	\draw [fill=black] (1,1) circle  (.08)  node [right]  at (0.8,0.6) {{$y(= x_1)$}};
	\draw [fill=black] (1,5) circle  (.08)  node [above]  at (1,5) {{$x$}}; 
	\draw [fill=black] (1.5,4) circle  (.08) node [left]  at (1.5,4) {{$x_2$}};
	\draw [fill=black] (2,4) circle  (.08) ;
	\draw [fill=black] (2.5,4) circle  (.08);
	\draw [fill=black] (3,4) circle  (.08) node [right]  at (3,4) {{$x_{2n-3}$}};

	\draw  [style = dotted] plot [smooth, tension=1.5] coordinates { (1,1) (1,5) };
	\draw  plot [smooth, tension=1.5] coordinates { (1,5) (1.5,4) };
	\draw  plot [smooth, tension=1.5] coordinates { (1,5) (2,4) };
	\draw  plot [smooth, tension=1.5] coordinates { (1,5) (2.5,4) };
	\draw  plot [smooth, tension=1.5] coordinates { (1,5) (3,4) };

	\draw plot [smooth, tension=1.5] coordinates { (1,1) (1.5 , 2.7) (1.5,4)};
	\draw plot [smooth, tension=1.5] coordinates {  (1,1) (2 , 2.7) (2,4)};
	\draw plot [smooth, tension=1.5] coordinates {   (1,1) (2.2 , 2.7) (2.5,4)};
	\draw plot [smooth, tension=1.5] coordinates {   (1,1) (2.5 , 2.7) (3,4)};


	\draw [fill=black] (8,5) circle  (.08)  node [above]  at (8,5) {$x^h$}; 
	\draw [fill=black] (8.5,4) circle  (.08);
	\draw [fill=black] (7.5,4) circle  (.08);
	\draw [fill=black] (7,4) circle  (.08) node [left]  at (7,4) {{$x_1^h (= y^h)$}};
	\draw [fill=black] (9,4) circle  (.08) node [right]  at (9,4) {{$x_{2n-3}^h$}};
	\draw [fill=black] (8,4) circle  (.08);
	
	\draw [fill=black] (8,1) circle  (.08)  node [below]  at (8,1) {$z$};
	
	\draw [style = dotted] plot [smooth, tension=1.5] coordinates { (8,5) (8.5,4) };
	\draw [style = dotted] plot [smooth, tension=1.5] coordinates {  (8,5) (7.5,4) };
	\draw  [style = thick] [blue] plot [smooth, tension=1.5] coordinates {   (8,5) (7,4) };
	\draw [style = dotted] plot [smooth, tension=1.5] coordinates {   (8,5) (9,4) };
	\draw [style = dotted] plot [smooth, tension=1.5] coordinates {  (8,5) (8,4) };
	
	\draw plot [smooth, tension=1.5] coordinates { (8.5,4) (8.5 , 2.5) (8,1)};
	\draw plot [smooth, tension=1.5] coordinates {  (7.5,4) (7.5 , 2.5) (8,1)};
	\draw [style = thick] [blue] plot [smooth, tension=1.5] coordinates {   (7,4) (7 , 2.5) (8,1)};
	\draw plot [smooth, tension=1.5] coordinates {   (9,4) (9 , 2.5) (8,1)};
	\draw plot [smooth, tension=1.5] coordinates {   (8,4) (8 , 2.5) (8,1)};


	\draw  plot [smooth, tension=1.5] coordinates { (2.5,4) (6, 3) (8.5,4) };
	\draw  plot [smooth, tension=1.5] coordinates {  (1.5,4) (5, 3) (7.5,4) };
	\draw [style = thick] [blue] plot [smooth, tension=1.5]  coordinates {   (1,1) (4.5, 2.5) (7,4) };
	\draw  plot [smooth, tension=1.5] coordinates {   (3,4) (6.5, 3) (9,4) };
	\draw  plot [smooth, tension=1.5] coordinates {  (2,4) (5.5, 3) (8,4) };
	\draw [style = thick] [blue] plot [smooth, tension=1.5] coordinates {  (1,5) (4.5, 5.5) (8,5) };

	\draw node [below] at (2,-0.5) {$AQ_{n-1}^0$};
	\draw node [below] at (8,-0.5) {$AQ_{n-1}^1$};
	
	\end{tikzpicture}
	
	
	\hspace{0.1 in} {$ \mathrm{Figure}\,\, 13.$}
	
\end{center}

In the same way, we get $2n - 3$ pendant $S$-Steiner trees if $z$ is adjacent to $x^h$ or $x^c$ or both but not adjacent to $y^h, y^c$.\\

{\bf Subcase 2.2.3(c) :} If $z$ is adjacent to $y^h$ only or to $y^h, x^c, x^h$ then use a path cover of order $2n -3$ in between $y^c$ instead of $y^h$ and $z$ in $AQ_{n-1}^1$ and get the required result as similar to Subcase 2.2.3(a).\\ Similarly We get $2n - 3$ pendant $S$-Steiner trees if  $z$ is adjacent to $x^h$ only or to $x^h, y^c, y^h.$ 

Hence by the principle of mathematical induction, we get that $\tau_{3}(AQ_n) = 2n - 3$.
\end{proof}

\section{Concluding Remarks:}
In this paper, we get pendant $3$-tree connectivity of $AQ_n$ i.e. $\tau_3(AQ_n) = 2n -3.$ Evaluations of $\tau_k(AQ_n)$ for $k \geq 4$ are still open.

\centerline{\textbf{Acknowledgement}}
The first author gratefully acknowledges the Department of Science and Technology, New Delhi, India for the award of Women Scientist Scheme(SR/WOS-A/PM-79/2016) for research in Basic/Applied Sciences.

\end{document}